%% file: regionchoice.tex
\documentclass[a4paper,12pt]{amsart}

\usepackage{amssymb}
\usepackage{amsmath}
\usepackage{graphicx}

\title{An integral region choice problem on knot projection}

\author{Kazushi Ahara}
\address{Department of Mathematics, Meiji University, 1-1-1 Higashi-Mita, Tama-ku, Kawasaki, Kanagawa, 214-8571, Japan}
\email{kazuaha63@hotmail.co.jp}

\author{Masaaki Suzuki}
\address{Department of Mathematics, 
Akita University, 
1-1 Tegata-Gakuenmachi, Akita, 010-8502, Japan}
\email{macky@math.akita-u.ac.jp}


\newtheorem{thm}{Theorem}[section]
\newtheorem{prop}[thm]{Proposition}
\newtheorem{lem}[thm]{Lemma}
\newtheorem{cor}[thm]{Corollary}

\theoremstyle{definition}
\newtheorem{definition}[thm]{Definition}
\newtheorem{example}[thm]{Example}
\newtheorem{remark}[thm]{Remark}

\begin{document}

\begin{abstract}
 In this paper we propose {\it a region choice problem} for a knot projection.  This problem is an integral extension of Shimizu's \lq region crossing change unknotting operation.'  We show that there exists a solution of the region choice problem for all knot projections.
\end{abstract}

\maketitle

\input{./introduction.tex}

\input{./preliminary.tex}

\input{./double.tex}

\input{./single.tex}

\input{./appendix.tex}

\input{./acknowledgements.tex}

\bibliographystyle{amsplain}

\end{document}

%% file: introduction.tex

\section{Introduction}

Let $K$ be a knot and $D$ a digram of $K$.  
The areas surrounded by arcs are called {\it regions}. 
Shimizu \cite{shimizu} defined a {\it region crossing change} 
at a region $R$ to be the crossing change at all the crossings on $\partial R$ 
and showed that this local transformation is 
an unknotting operation.  
We can interpret this problem as follows. 
Let $\bar{D}$ be a projection of a knot $K$, 
where $\bar{D}$ possesses some arcs and crossings. 
We suppose that each crossing has already been equipped with 
a point $0$ or $1$ modulo $2$ and 
that if we choose a region $R$, then the points of all the crossings 
which lie on $\partial R$ are increased $1$ modulo $2$. 
Shimizu \cite{shimizu} showed that 
if we choose some regions appropriately, 
then the points of all the crossings become $0$. 
Here we call it a region choice problem modulo $2$. 
Namely, the region choice problem modulo $2$ is solvable. 

In this paper, we extend this region choice problem to an integral range. 
That is, each crossing has been equipped with an integral point and we assign an integer to each region in order to make all points on crossings $0$.
We may consider two rules, the {\it single counting rule} and the {\it double counting rule}. 
If we assign an integer $u$ to a region $R$ with the single counting rule, 
the points of all the crossings which lie on the boundary of $R$ are increased $u$. 
On the other hand, 
if we assign an integer $u$ to a region $R$ with the double counting rule, 
the points of the crossings which the region $R$ touches twice are increased $2 u$ 
and the points of the crossings which the region $R$ touches once are increased $u$. 
See an example in Figure \ref{fig:1-1}. 
\begin{figure}[t]
\includegraphics[]{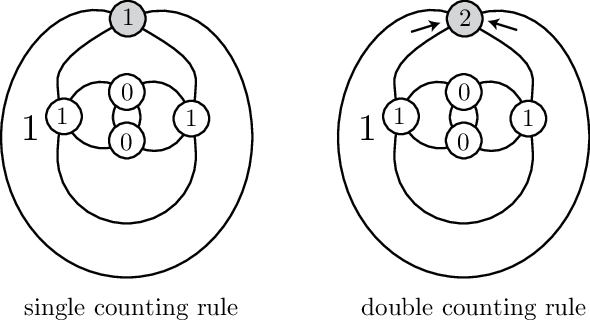}
\caption{}\label{fig:1-1}
\end{figure}
We suppose that each crossing has already been equipped with an integer
 as an initial setting.   
In this paper, we show that if we assign an integer to each region appropriately, 
the points of all the crossings become $0$. 
See an example in Figure \ref{fig:1-2}. 
\begin{figure}
\includegraphics[]{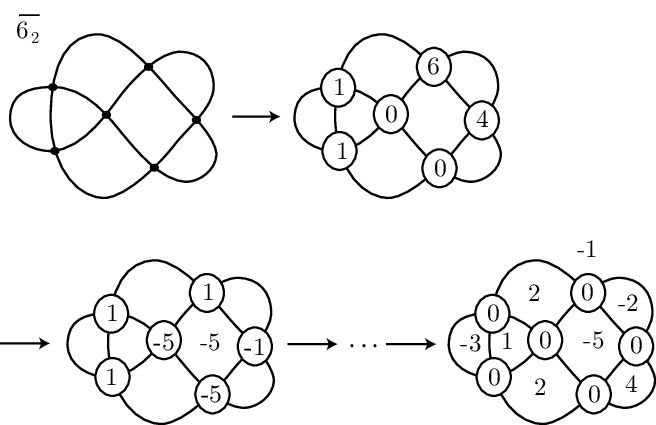}
\caption{}\label{fig:1-2}
\end{figure}

 This paper is organized as follows.  In Section $2$, we introduce some notations.  Here we define an integral matrix $A_i(\bar D)$, called {\it a region choice matrix}, as a coefficient matrix of the system of equations of region choice problem.  Cheng and Gao proposed {\it an incidence matrix} in \cite{cheng-gao}, this is a modulo $2$ reduction of $A_1(\bar D)$.  In Section $3$, $4$, we show the main result of this paper.  Section $5$, $6$, $7$ are  appendices.  In Section 5, we show another proof of the main result for the double counting rule.  In Section 7, we show solutions of region choice problem for knot projections from $\bar 3_1$ to $\bar 6_3$

%% file: preliminary.tex

\section{Preliminary}

In this section, we define some notations and show some basic facts. 
First an easy argument gives the following. 

\begin{lem}
Let $\bar{D}$ be a knot (or link) projection. 
We denote by $m$ and $n$ 
the number of regions and crossings in $\bar{D}$ respectively. 
Then we have $m = n + 2$. 
\end{lem}

\begin{proof}
We can regard the projection $\bar{D}$ as a graph on a sphere $S^2$ 
by considering arcs and crossings to be edges and vertices respectively. 
Since each crossing is joined by $4$ arcs, the number of arcs is $2 n$. 
On the other hand, $m + n - 2$ is equal to the number of arcs 
by the fact that the euler characteristic of $S^2$ is 2. 
Therefore we obtain the statement. 
\end{proof}

For a knot projection, we define a {\it region choice matrix} 
which has information about the relationship between regions and crossings. 

\begin{definition}[region choice matrix]
Let $\bar{D}$ be a knot projection. 
We denote by $R=\{r_1,r_2,\cdots,r_{n+2}\}$ and $V=\{v_1,v_2,\cdots, v_n\}$ 
the set of regions and crossings respectively. 

(1) The region choice matrix of the single counting rule 
\[
A_1(\bar{D})=\left( a^{(1)}_{ij} \right)_{(i,j)} \in M_{n,n+2}(\mathbb{Z})
\]
is determined by 
\[
a^{(1)}_{ij}= 
\left\{
\begin{array}{ll}
1 & (\mbox{if the crossing $v_i$ lies on the boundary of the region $r_j$}) \\
0 & (\mbox{otherwise}).
\end{array}
\right.
\]

(2) The region choice matrix of the double counting rule 
\[
A_2(\bar{D})= \left( a^{(2)}_{ij} \right)_{(i,j)} \in M_{n,n+2}(\mathbb{Z})
\]
is determined by  
\[
a^{(2)}_{ij}= 
\left\{
\begin{array}{ll}
2 & (\mbox{if the region $r_j$ touches the crossing $v_i$ twice})\\
1 & (\mbox{if the region $r_j$ touches the crossing $v_i$ once})\\
0 & (\mbox{otherwise}).
\end{array}
\right.
\]
\end{definition}

For a later argument, let us introduce the following terminology. 

\begin{definition}[kernel solution]\label{def:2-3}
Let $\bar{D}$ be a knot projection and $A_i=A_i(\bar{D})$ the region choice matrix of $\bar{D}$. 
We write $n$ for the number of regions. 
A vector $\boldsymbol{u} \in \mathbb{Z}^{n+2}$ 
is called a kernel solution (of the single/double counting rule)
if $A_i \boldsymbol{u} = \boldsymbol{o}$. 
\end{definition}

We now show an example 
about the single and double counting rules. 

\begin{example}\label{ex:2-4}
Consider a knot projection $\bar{D}$ as shown in Figure \ref{fig:2-1}. 
This knot projection gives us the region choice matrices 
\begin{figure}
\includegraphics[]{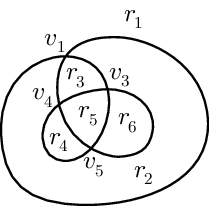}
\caption{Example \ref{ex:2-4}}\label{fig:2-1}
\end{figure}
\[
A_1 (\bar{D}) = 
\left(
\begin{array}{cccccc}
 1 & 1 & 1 & 0 & 0 & 0 \\
 0 & 1 & 1 & 1 & 1 & 0 \\
 0 & 1 & 1 & 0 & 1 & 1 \\
 0 & 1 & 0 & 1 & 1 & 1 
\end{array}
\right), ~
A_2 (\bar{D}) = 
\left(
\begin{array}{cccccc}
 1 & 2 & 1 & 0 & 0 & 0 \\
 0 & 1 & 1 & 1 & 1 & 0 \\
 0 & 1 & 1 & 0 & 1 & 1 \\
 0 & 1 & 0 & 1 & 1 & 1 
\end{array}
\right). 
\]
A nontrivial vector 
\[
\boldsymbol{u} = 
\left(
\begin{array}{c}
1 \\ -2 \\ 1 \\ 1 \\ 0 \\ 1
\end{array}
\right)
\in {\mathbb Z}^6
\]
is a kernel solution of the single counting rule. 
However, 
$\boldsymbol{u}$ is not a kernel solution of the double counting rule. 
\end{example}

%% file: double.tex

\section{Double counting rule}

In this section, 
we consider the region choice problem 
with the double counting rule. 

\begin{thm}[region choice problem of the double counting rule is solvable]\label{thm:3-1}
For any knot projection $\bar{D}$, 
let $A_2=A_2(\bar{D})\in M_{n,n+2}(\mathbb{Z})$ be the region choice matrix 
of the double counting rule 
and $\boldsymbol{b}\in \mathbb{Z}^n$ a given integral vector. 
Then there exists a solution $\boldsymbol{u}\in \mathbb{Z}^{n+2}$ such that 
\[
A_2 \boldsymbol{u} + \boldsymbol{b} = \boldsymbol{o}.
\]
\end{thm}

Theorem \ref{thm:3-1} shows the existence 
of a solution of the region choice problem of the double counting rule. 
We prove the following lemma in order to show Theorem \ref{thm:3-1}. 

\begin{lem}\label{lem:3-2}
Fix an arc $\gamma$ in the knot projection $\bar{D}$ and 
let $r,r'$ be two regions which are the both sides of the arc $\gamma$. 
For any integers $a,b$, 
there exists a kernel solution $\boldsymbol{u} \in {\mathbb Z}^{n+2}$ 
of the double counting rule 
such that ${u}^{(r)}=a$ and ${u}^{(r')}=b$.
Here ${u}^{(r)},{u}^{(r')}$ are elements of $\boldsymbol{u}$ 
corresponding to the regions $r,r'$ respectively. 
\end{lem}

\begin{proof}
We take a point $p$ on the arc $\gamma$ as shown in Figure \ref{fig:3-0}. 
\begin{figure}
\includegraphics[]{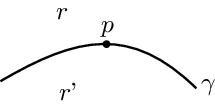}
\caption{}\label{fig:3-0}
\end{figure}

For any knot diagram $D$, after some crossing changes, we can make $D$ into a diagram $D'$ of a trivial knot.  
Remark that the projections $\bar D$ and $\bar D'$ coincide.  Then for any knot projection $\bar D$, 
we may regard $\bar{D}$ as a knot projection of a trivial knot.

By using Reidemeister moves (RI), (RII), (RIII) for knot projections as in Figures \ref{fig:R1}, \ref{fig:R2}, \ref{fig:R3}, 
we can transpose $\bar D$ to $\bar D_0$  in Figure \ref{fig:5-1}.  
Here we consider knot projections with at least one crossing, then the projection in Figure \ref{fig:5-1} is a minimal configuration.
\begin{figure}
\includegraphics[]{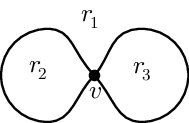}
\caption{$\bar D_0$}\label{fig:5-1}
\end{figure}

But this observation is not enough for our lemma.  
In order to avoid passing segments through $p$ while transposing $\bar{D}$ to $\bar{D}_0$, we may make segments pass through the infinity point.  
That is, we add another move (RIV), as in Figure \ref{fig:R4}.
Therefore it is sufficient to show that 
if a knot projection satisfies the statement,  
then another knot projection which is obtained from 
(RI), (RII), (RIII) and (RIV) fixing the neighborhood of $p$ 
also satisfies the statement. 

(RI) 
If one of the left projections in Figure \ref{fig:R1} satisfies the statement, 
then the right projection also satisfies the statement
by assigning each integer as shown in the right projection. 
\begin{figure}
\includegraphics[]{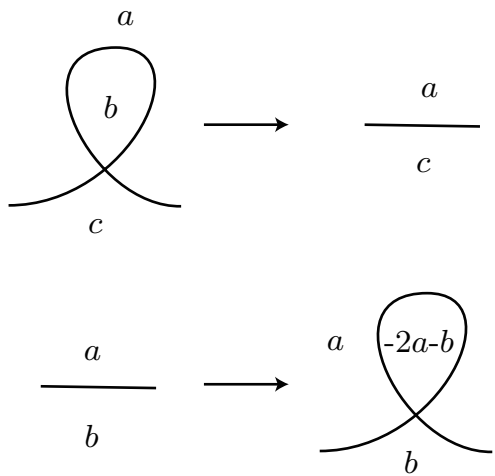}
\caption{Reidemeister move I}\label{fig:R1}
\end{figure}

(RII) 
If one of the left projections in Figure \ref{fig:R2} satisfies the statement, 
then the right projection also satisfies the statement
by assigning each integer as shown in the right projection. 
We remark that $a + b + d + e = 0,  a + c + d + e = 0$ 
in the upper left projection, 
since this projection satisfies the statement. 
Therefore we have $b = c$. 
\begin{figure}
\includegraphics[]{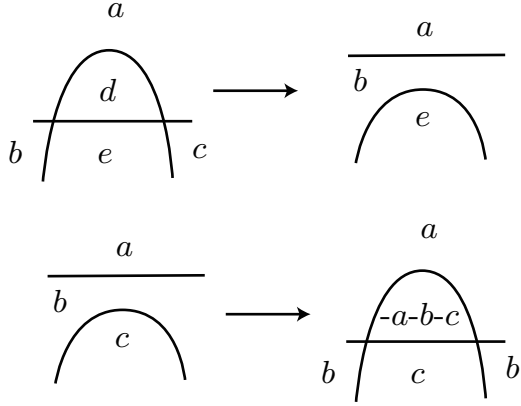}
\caption{Reidemeister move II}\label{fig:R2}
\end{figure}

(RIII) 
Suppose that the left projection in Figure \ref{fig:R3} satisfies the statement. 
\begin{figure}
\includegraphics[]{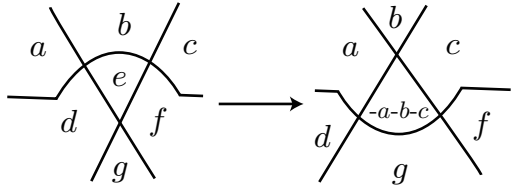}
\caption{Reidemeister move III}\label{fig:R3}
\end{figure}
Then we have the equalities 
\[
a + b + d + e = 0, ~ 
b + c + e + f = 0, ~
d + e + f + g = 0. 
\]
We assign $- a - b - c$ to the center region. 
At the three crossings, the following equalities hold:  
\begin{align*}
& a + b + c + ( - a - b - c) = 0, \\
& a + (- a - b - c) + d + g = d + g - b - c = d + g + e + f = 0, \\
& c + (- a - b - c) + f + g = f + g - a - b = f + g + d + e = 0. 
\end{align*}

(RIV) 
If the left projections in Figure \ref{fig:R4} satisfies the statement, 
then the right projection also satisfies the statement
by assigning each integer as shown in the right projection. 
\begin{figure}
\includegraphics[]{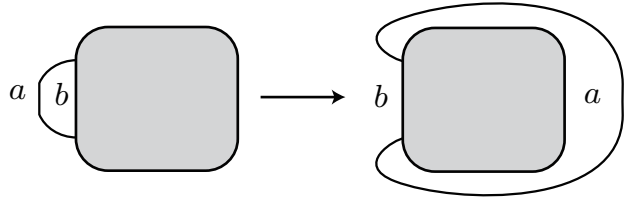}
\caption{passing through the infinity point}\label{fig:R4}
\end{figure}
Hence the right projection also satisfies the statement 
and this completes the proof. 
\end{proof}

We remark that Lemma \ref{lem:3-2} also holds for links by similar argument. 
By making use of Lemma \ref{lem:3-2}, 
we obtain the following proposition. 

\begin{prop}[add-$1$ operation]\label{prop:3-3}
Let $K$ be a knot and $\bar{D}$ a knot projection. 
We take a crossing $v$ in $\bar{D}$. 
There exists $\boldsymbol{u} \in \mathbb{Z}^{n+2}$ 
such that any elements of $A_2 \boldsymbol{u}$ are zero 
but the element of $A_2 \boldsymbol{u}$ corresponding to $v$ is $1$.
\end{prop}

\begin{proof}
First we give $K$ an orientation and splice it at the crossing $v$. 
Then we obtain a two-component link, 
whose components are denoted by $L_1,L_2$. 
We assign $0$ to the region of $v$ and $1$ to the adjacent region of $L_1$ 
as shown in Figure \ref{fig:3-4}. 
\begin{figure}
\includegraphics[]{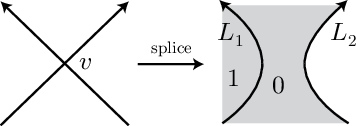}
\caption{}\label{fig:3-4}
\end{figure}
By Lemma \ref{lem:3-2}, we have a kernel solution for $L_1$, 
that is, there exists a certain assignment some integers 
to the other regions of $L_1$ by ignoring $L_2$ such that 
the assignment does not change the points 
of all the crossings. 
Next, we apply a checkerboard coloring for $L_2$ by ignoring $L_1$. 
Then the new assignment in the white region is 
$-1$ times the integer which is determined 
by the above assignment of $L_1$. 
The assignment in the black region is not changed 
from that of $L_1$. 
Finally we unsplice $L_1$ and $L_2$ at the crossing $v$ and 
get a assignment of $\bar{D}$. 

We check how the point of each crossing is changed by this assignment. 
The point of $v$ is increased $1$. 
The points of the self-crossings of $L_1$ are not changed 
by Lemma \ref{lem:3-2}, as shown in Figure \ref{fig:L1L1}.
\begin{figure}
\includegraphics[]{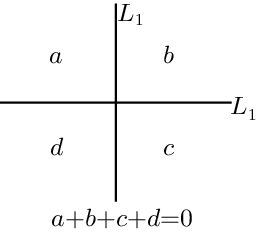}
\caption{}\label{fig:L1L1}
\end{figure}
All the four regions which touch a self-crossing of $L_2$ 
are assigned the same integer. 
Then the points of such crossings are not changed  
by the checkerboard coloring, as shown in Figure \ref{fig:L2L2}. 
\begin{figure}
\includegraphics[]{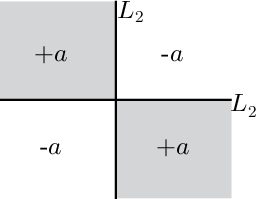}
\caption{}\label{fig:L2L2}
\end{figure}
In the four regions 
which touch the crossings of $L_1$ and $L_2$, 
two regions are assigned the same integer
and the other two regions are so. 
One region is colored black and 
the other region is colored white 
in these two pairs. 
Then the points of such crossings are not changed, as shown in Figure \ref{fig:L1L2}. 
\begin{figure}
\includegraphics[]{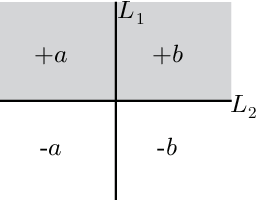}
\caption{}\label{fig:L1L2}
\end{figure}
Therefore this assignment satisfies the statement 
and completes the proof. 
\end{proof}

In Appendix II, we show an example of the procedure 
in the proof of Proposition \ref{prop:3-3} .

Proposition \ref{prop:3-3} gives us 
how to increase $1$ at an arbitrary crossing 
without changing the points of any other crossings. 
Hence this proves Theorem \ref{thm:3-1}.

%% file: single.tex

\section{single counting rule case}

In this section we discuss the region choice problem with the single counting rule.  In the original problem (Shimizu's region crossing change,) counting rule is the single one and if we consider the region choice problem (of this rule) in the context of modulo 2, the solution $\boldsymbol{u}$ of 
$$A_1\boldsymbol{u}+ \boldsymbol{b}\equiv\boldsymbol{o}\quad (\text{mod.}\ 2)$$ i$\boldsymbol{b}$ is an arbitrary vector) gives a solution of original unknotting problem of the region crossing change, \cite{shimizu}.

Our second result is stated as follows:

\begin{thm}[region choice problem of the single counting rule is solvable]\label{thm:4-1}
For any knot projection ${\bar D}$, let $A_1=A_1({\bar D})\in M_{n,n+2}(\mathbb{Z})$ be the region choice matrix of the single counting rule 
and $\boldsymbol{b}\in \mathbb{Z}^n$ a given integral vector.  Then there exists a solution $\boldsymbol{u}\in \mathbb{Z}^{n+2}$ such that 
$$A_1\boldsymbol{u}+ \boldsymbol{b}=\boldsymbol{o}.$$
\end{thm}

To prove this theorem, first we consider the single counting rule's version of Lemma \ref{lem:3-2}.  

\begin{lem}\label{lem:4-2}
  Fix an arc $\gamma$ in the knot projection ${\bar D}$ 
and let $r,r'$ be two regions which are the both sides of the arc $\gamma$.  
For any integers $a,b$,  there exists a kernel solution $\boldsymbol{u} \in {\mathbb Z}^{n+2}$ 
of the single counting rule 
such that ${u}^{(r)}=a$ and ${u}^{(r')}=b$.
Here ${u}^{(r)},{u}^{(r')}$ are elements of $\boldsymbol{u}$ 
corresponding to the regions $r,r'$ respectively. 
\end{lem}

\begin{proof}
We use induction on the number of reducible crossings in ${\bar D}$.

First we consider the case ${\bar D}$ is irreducible.  
In this case, the two region choice matrices $A_1({\bar D})$ and $A_2({\bar D})$ coincide.  
It follows that we have a solution by Lemma \ref{lem:3-2}. 

Next, we assume that if the number of reducible crossing is less than $k$ then there exists a kernel solution.  
Suppose that knot projection ${\bar D}$ has $k$ reducible crossings 
(Figure \ref{fig:4-1}). 
We choose a reducible crossing $v$.

\begin{figure}
\includegraphics[]{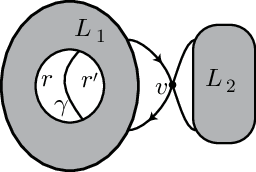}
\caption{}\label{fig:4-1}
\end{figure}

We splice ${\bar D}$ at the crossing $v$ and suppose that we obtain two component link ${L}_1 \cup {L}_2$ (Figure \ref{fig:4-2}).  
(Because $v$ is a reducible crossing, we have two connected components after splicing.)   
Here let ${L}_1$ be the component with the arc $\gamma$.  

\begin{figure}
\includegraphics[]{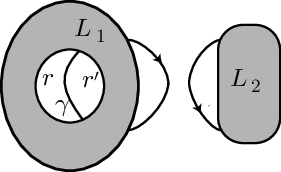}
\caption{}\label{fig:4-2}
\end{figure}

The number of reducible crossings of ${L}_1$ is less than $k$. 
By the assumption of induction, there exists a kernel solution $\boldsymbol{u}_1$ of $A_1({L}_1)$ 
such that ${u}_1^{(r)}=a$ and ${u}_1^{(r')}=b$ (Figure \ref{fig:4-3}). 
Here let $c,d$ be integers which appear in $\boldsymbol{u}_1$ correspondent with the regions in Figure \ref{fig:4-3}.

\begin{figure}
\includegraphics[]{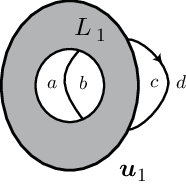}
\caption{}\label{fig:4-3}
\end{figure}

The number of reducible crossings in ${L}_2$ is less than $k$, we apply the assumption of induction on ${L}_2$ with the condition as in Figure \ref{fig:4-4} and we have a kernel solution $\boldsymbol{u}_2$ of $A_1({L}_2)$.

\begin{figure}
\includegraphics[]{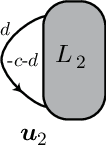}
\caption{}\label{fig:4-4}
\end{figure}

We merge $\boldsymbol{u}_1$ and $\boldsymbol{u}_2$ in ${\bar D}$ 
and we have a kernel solution $\boldsymbol{u}$ such that ${u}^{(r)}=a$ and ${u}^{(r')}=b$ (Figure \ref{fig:4-5}).  

\begin{figure}
\includegraphics[]{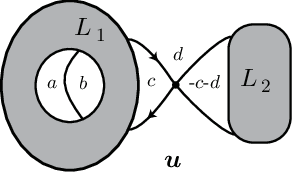}
\caption{}\label{fig:4-5}
\end{figure}

Indeed, for any crossings $v_i$ in ${L}_1$, the element $v_i$ of $A_1({\bar D}) \boldsymbol{u}$ vanishes because $\boldsymbol{u}_1$ is a kernel solution of $A_1({L}_1)$.  
In the same way for any crossings $v_i$ in ${L}_2$, the element $v_i$ of $A_1({\bar D}) \boldsymbol{u}$ vanishes.  
At the crossing $v$, the element $v$ of $A_1({\bar D}) \boldsymbol{u}$ is $c+d+(-c-d)=0$.  This completes the proof.
\end{proof}

Using the same idea of the proof of this lemma and of Proposition \ref{prop:3-3}, we have a corollary.

\begin{cor}[add-1 operation at a reducible crossing]\label{cor:4-3}
Let ${\bar D}$ be a knot projection and $v$ a reducible crossing of ${\bar D}$.  
There exists $\boldsymbol{u}_v\in\mathbb{Z}^{n+2}$ such that any element of $A_1\boldsymbol{u}_v$ are zero but the element $v$ is $1$.
\end{cor}

\begin{proof}
In the proof of Lemma \ref{lem:4-2}, we replace $\boldsymbol{u}_2$ by $\boldsymbol{u}'_2$ with condition in Figure \ref{fig:4-6}.  Then merged vector $\boldsymbol{u}_v$ of $\boldsymbol{u}_1$ and $\boldsymbol{u}'_2$ gives our solution. \end{proof}

\begin{figure}
\includegraphics[]{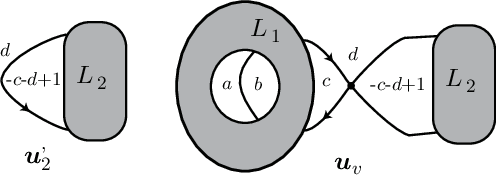}
\caption{}\label{fig:4-6}
\end{figure}

\begin{proof}[Proof of Theorem \ref{thm:4-1}]
We will show Theorem \ref{thm:4-1}.
A difference between region choice matrices $A_1=A_1({\bar D})$ and $A_2=A_2({\bar D})$ is how to deal with a region $r_j$ which touches a crossing $v_i$ twice.  In this case, $(i,j)$ component of $A_1({\bar D})$ is one and that of $A_2({\bar D})$ is two.

Now, for each region $r_j$, let $v_1^{(j)},v_2^{(j)},\cdots,v_{s_j}^{(j)}$ be crossings which are touched by $r_j$ twice.  (Here $s_j$ might be zero.)  See Figure \ref{fig:4-7} .

\begin{figure}
\includegraphics[]{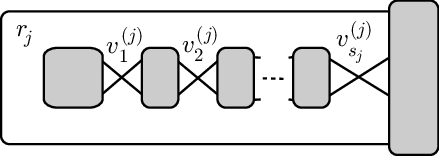}
\caption{}\label{fig:4-7}
\end{figure}

If we let 
$$\boldsymbol{r}_j=\begin{pmatrix}0\\ \vdots \\ 0 \\ 1 \\ 0 \\ \vdots \\ 0 \end{pmatrix} r_j,\quad \boldsymbol{v}_i^{(j)}=\begin{pmatrix}0\\ \vdots \\ 0 \\ 1 \\ 0 \\ \vdots \\ 0 \end{pmatrix} v_i^{(j)},$$
then we have 
$$ A_2 \boldsymbol{r}_j - A_1 \boldsymbol{r}_j=\sum_{i=1}^{s_j}\boldsymbol{v}_i^{(j)}$$

By Corollary \ref{cor:4-3}, there exists $\boldsymbol{u}_{v_i^{(j)}}$ such that 
$$ A_1\boldsymbol{u}_{v_i^{(j)}}=\boldsymbol{v}_i^{(j)}.$$
In Section 3, we show that for any $\boldsymbol{b}\in\mathbb{Z}^n$, there exists $\boldsymbol{u}\in\mathbb{Z}^{n+2}$ such that $A_2\boldsymbol{u}=-\boldsymbol{b}$.  Let $\boldsymbol{u}=\displaystyle\sum_{j} a_j\boldsymbol{r}_j$ and $\boldsymbol{u'}= \boldsymbol{u}+\displaystyle\sum_{j}a_j\sum_{i=1}^{s_j}\boldsymbol{u}_{v_i^{(j)}}$. 

Under these preparations, we obtain
\begin{align*}
A_1\boldsymbol{u'} &= \displaystyle\sum_{j} a_jA_1\boldsymbol{r}_j + \displaystyle\sum_{j}a_j\sum_{i=1}^{s_j}A_1\boldsymbol{u}_{v_i^{(j)}} \\
&=\displaystyle\sum_{j} a_j\left( A_1\boldsymbol{r}_j + \sum_{i=1}^{s_j}\boldsymbol{v}_i^{(j)}\right) \\
&=\displaystyle\sum_{j} a_j A_2\boldsymbol{r}_j = A_2\boldsymbol{u}=-\boldsymbol{b}
\end{align*}

This completes the proof of Theorem \ref{thm:4-1}. 
\end{proof}

%% file: appendix.tex

\section{Appendix I}

In this section, we give another proof of Theorem \ref{thm:3-1} using only matrix elementary operations.  
This proof might be simpler than that of in Section $3$.  
The proof in Section $3$ is surely geometric and the proof in this section is algebraic, and both proofs are valuable for studying this problem.

First we review integral elementary operations. 
Two integral matrices $A,B \in M_{n,n+2}(\mathbb{Z})$ are 
called to be $\mathbb{Z}$-equivalent 
if we transpose $A$ into $B$ by the followings.
\begin{itemize}
\item Interchange two rows (resp. two columns.)
\item Multiply a row (resp. a column) by $-1$.
\item Add a row (resp. a column) to another one multiplied by an integer. 
\end{itemize} 
Remark that these elementary operation matrices are of determinant $\pm 1$ 
and that the inverse is also an integral matrix.  

\begin{lem}\label{lem:5-1}
If an integral matrix $A \in M_{n,n+2}(\mathbb{Z})$ is $\mathbb{Z}$-equivalent to 
$$ E_{00}:=\left(\begin{array}{ccc|cc} 1&&0&0&0 \\ &\ddots &&\vdots&\vdots \\ 0&&1&0&0 \end{array}\right), $$
then for any $\boldsymbol{b}\in\mathbb{Z}^n$ 
there exists a solution $\boldsymbol{u}\in \mathbb{Z}^{n+2}$ 
such that $A\boldsymbol{u}+\boldsymbol{b}=\boldsymbol{o}$. 
\end{lem}

\begin{proof}
Since $A$ is $\mathbb{Z}$-equivalent to $E_{00}$, 
there exist integral non-singular matrices $P,Q$ such that 
$$ PAQ=E_{00} .$$
Indeed, $P,Q$ are certain products of elementary operation matrices. Set 
$$ \boldsymbol{u}=Q\begin{pmatrix}-P\boldsymbol{b} \\ \alpha \\ \beta \end{pmatrix} ,$$
where $\alpha,\beta$ are any integers.  Then we have 
\begin{align*}
A\boldsymbol{u}&=AQ\begin{pmatrix}-P\boldsymbol{b} \\ \alpha \\ \beta \end{pmatrix}
=P^{-1}PAQ\begin{pmatrix}-P\boldsymbol{b} \\ \alpha \\ \beta \end{pmatrix}\\
&=P^{-1}E_{00}\begin{pmatrix}-P\boldsymbol{b} \\ \alpha \\ \beta \end{pmatrix} 
=P^{-1}(-P\boldsymbol{b})=-\boldsymbol{b}.
\end{align*}
\end{proof}

Due to Lemma \ref{lem:5-1}, next we will consider that 
a region choice matrix is $\mathbb{Z}$-equivalent to $E_{00}$. 

\begin{thm}[Therem \ref{thm:3-1}]\label{thm:5-2}
For any knot projection ${\bar D}$, 
the region choice matrix of the double counting rule $A_2({\bar D})$ 
is $\mathbb{Z}$-equivalent to $E_{00}$.
\end{thm}

Before proving Theorem \ref{thm:5-2}, 
we will discuss how Reidemeister moves affect region choice matrices. 

\begin{prop}[Reidemeister move I]\label{prop:5-3}
Let ${\bar D}_1,{\bar D}_1'$ are knot projections as in Figure \ref{fig:5-3}. 
Then $A_2({\bar D}_1)\oplus (1)$ and $A_2({\bar D}_1')$ are $\mathbb{Z}$-equivalent to each other. 

\begin{figure}
\includegraphics[]{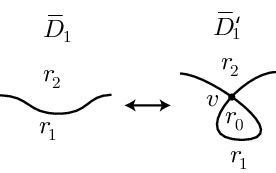}
\caption{}\label{fig:5-3}
\end{figure}

\end{prop}

\begin{proof}
If we represent $A_2({\bar D}_1)$ by 
$$ A_2({\bar D}_1)=\begin{pmatrix} \boldsymbol{a}&  \boldsymbol{b}&P\end{pmatrix}, $$
where the columns of $\boldsymbol{a},\boldsymbol{b}$ in $A_2({\bar D}_1)$ 
correspond to the regions $r_1,r_2$ respectively, 
then $A_2({\bar D}_1')$ is 
$$ A_2({\bar D}_1')=\left(\begin{array}{cc|c|c}  \boldsymbol{a}&  \boldsymbol{b} & P & \boldsymbol{o} \\ \hline 2&1&0&1 \end{array}\right). $$
The last column of $A_2({\bar D}_1')$ corresponds to $r_0$ 
and the last row corresponds to $v$.  
Adding the column $r_0$ to the column $r_1$ (resp. the column $r_2$) 
multiplied by $-2$ (resp. $-1$), 
we get  $A_2({\bar D}_1)\oplus (1)$ from $A_2({\bar D}_1')$ by integral elementary operations: 
$$ \left(\begin{array}{cc|c|c}  \boldsymbol{a}&  \boldsymbol{b} & P & \boldsymbol{o} \\ \hline 2&1&0&1 \end{array}\right)\longrightarrow \left(\begin{array}{cc|c|c}  \boldsymbol{a}&  \boldsymbol{b} & P & \boldsymbol{o} \\ \hline 0&0&0&1 \end{array}\right). $$
\end{proof}

\begin{prop}[Reidemeister move II]\label{prop:5-4}
Let ${\bar D}_2,{\bar D}_2'$ are knot projections as in Figure \ref{fig:5-4}.  
Then $A_2({\bar D}_2)\oplus (1)\oplus (1) $ and $A_2({\bar D}_2')$ 
are $\mathbb{Z}$-equivalent to each other. 

\begin{figure}
\includegraphics[]{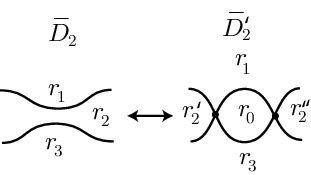}
\caption{}\label{fig:5-4}
\end{figure}

\end{prop}

\begin{proof}
If we represent $A_2({\bar D}_2)$ by 
$$ A_2({\bar D}_2)=\begin{pmatrix} \boldsymbol{a}&  \boldsymbol{b}&  \boldsymbol{c} & P \end{pmatrix}, $$
where the columns of $\boldsymbol{a},\boldsymbol{b},\boldsymbol{c}$ in $A_2({\bar D}_2)$ 
correspond to the regions $r_1, r_2, r_3$ respectively, 
then $A_2({\bar D}_2')$ is 
$$ A_2({\bar D}_2')=\left(\begin{array}{ccc|c|cc}\boldsymbol{a}&  \boldsymbol{b}'&  \boldsymbol{c}& P & \boldsymbol{b}'' &\boldsymbol{o} \\ \hline 1&1&1&0&0&1 \\ 1&0&1&0&1&1 \end{array}\right). $$
While we transpose ${\bar D}_2$ into ${\bar D}_2'$, 
the region $r_2$ is divided into the two regions $r_2'$ and $r_2''$.  
In the double counting rule, the contribution of $r_2$ is the sum of those of $r_2'$ and $r''_2$, 
then $\boldsymbol{b}=\boldsymbol{b}'+\boldsymbol{b}''$.
Now we operate $A_2({\bar D}_2')$  as follows:
\begin{align*}
& \left( \begin{array}{ccc|c|cc}\boldsymbol{a}&  \boldsymbol{b}'&  \boldsymbol{c}& P & \boldsymbol{b}'' &\boldsymbol{o} \\ \hline 1&1&1&0&0&1 \\ 1&0&1&0&1&1 \end{array}\right) 
 &\longrightarrow & \left( \begin{array}{ccc|c|cc}\boldsymbol{a}&  \boldsymbol{b}&  \boldsymbol{c}& P & \boldsymbol{b}'' &\boldsymbol{o} \\ \hline 1&1&1&0&0&1 \\ 1&1&1&0&1&1 \end{array}\right)  \\
\longrightarrow & \left( \begin{array}{ccc|c|cc}\boldsymbol{a}&  \boldsymbol{b}&  \boldsymbol{c}& P & \boldsymbol{b}'' &\boldsymbol{o} \\ \hline 0&0&0&0&0&1 \\ 0&0&0&0&1&1 \end{array}\right) 
&\longrightarrow & \left( \begin{array}{ccc|c|cc}\boldsymbol{a}&  \boldsymbol{b}&  \boldsymbol{c}& P & \boldsymbol{b}'' &\boldsymbol{o} \\ \hline 0&0&0&0&0&1 \\ 0&0&0&0&1&0 \end{array}\right) \\
\longrightarrow & \left( \begin{array}{ccc|c|cc}\boldsymbol{a}&  \boldsymbol{b}&  \boldsymbol{c}& P & \boldsymbol{o} &\boldsymbol{o} \\ \hline 0&0&0&0&0&1 \\ 0&0&0&0&1&0 \end{array}\right) 
&\longrightarrow & \left( \begin{array}{ccc|c|cc}\boldsymbol{a}&  \boldsymbol{b}&  \boldsymbol{c}& P & \boldsymbol{o} &\boldsymbol{o} \\ \hline 0&0&0&0&1&0 \\ 0&0&0&0&0&1 \end{array}\right). \end{align*}
Thus we  get  $A_2({\bar D}_2)\oplus (1)\oplus (1)$ from $A_2({\bar D}_2')$ 
by integral elementary operations. 
\end{proof}

\begin{prop}[Reidemeister move III]\label{prop:5-5}
Let ${\bar D}_3,{\bar D}_3'$ are knot projections as in Figure \ref{fig:5-5}.  
Then $A_2({\bar D}_3)$ and $A_2({\bar D}_3')$ are $\mathbb{Z}$-equivalent to each other. 

\begin{figure}
\includegraphics[]{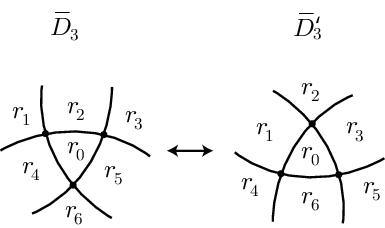}
\caption{}\label{fig:5-5}
\end{figure}

\end{prop}

\begin{proof}
If we represent $A_2({\bar D}_3)$ by 
$$ A_2({\bar D}_3)=\left(\begin{array}{ccccccc|c}\boldsymbol{a}&  \boldsymbol{b}&  \boldsymbol{c}&  \boldsymbol{d}&  \boldsymbol{e}&  \boldsymbol{f}&  \boldsymbol{o}& P \\ \hline 1&1&0&1&0&0&1&0 \\ 0&1&1&0&1&0&1&0 \\ 0&0&0&1&1&1&1&0 \end{array}\right), $$
where the columns of 
$\boldsymbol{a},\boldsymbol{b},\boldsymbol{c},\boldsymbol{d},\boldsymbol{e},\boldsymbol{f}$ 
in ${\bar D}_3$ correspond to the regions  $r_1, r_2, \ldots, r_6$ respectively, 
then $A_2({\bar D}_3')$ is 
$$ A_2({\bar D}_3')=\left(\begin{array}{ccccccc|c}\boldsymbol{a}&  \boldsymbol{b}&  \boldsymbol{c}&  \boldsymbol{d}&  \boldsymbol{e}&  \boldsymbol{f}&  \boldsymbol{o}& P \\ \hline 0&0&1&0&1&1&1&0 \\ 1&0&0&1&0&1&1&0 \\ 1&1&1&0&0&0&1&0 \end{array}\right). $$
The column with $\boldsymbol{o}$ is the column of $r_0$ in ${\bar D}_3$ and ${\bar D}_3'$. 
Similarly, we have 
\begin{align*}
& \left(\begin{array}{ccccccc|c}\boldsymbol{a}&  \boldsymbol{b}&  \boldsymbol{c}&  \boldsymbol{d}&  \boldsymbol{e}&  \boldsymbol{f}&  \boldsymbol{o}& P \\ \hline 1&1&0&1&0&0&1&0 \\ 0&1&1&0&1&0&1&0 \\ 0&0&0&1&1&1&1&0 \end{array}\right) \\
\longrightarrow & \left(\begin{array}{ccccccc|c}\boldsymbol{a}&  \boldsymbol{b}&  \boldsymbol{c}&  \boldsymbol{d}&  \boldsymbol{e}&  \boldsymbol{f}&  \boldsymbol{o}& P \\ \hline 0&0&-1&0&-1&-1&1&0 \\ -1&0&0&-1&0&-1&1&0 \\ -1&-1&-1&0&0&0&1&0 \end{array}\right)
\end{align*}
\begin{align*}
\longrightarrow & \left(\begin{array}{ccccccc|c}\boldsymbol{a}&  \boldsymbol{b}&  \boldsymbol{c}&  \boldsymbol{d}&  \boldsymbol{e}&  \boldsymbol{f}&  \boldsymbol{o}& P \\ \hline 0&0&1&0&1&1&-1&0 \\ 1&0&0&1&0&1&-1&0 \\ 1&1&1&0&0&0&-1&0 \end{array}\right)\\
\longrightarrow & \left(\begin{array}{ccccccc|c}\boldsymbol{a}&  \boldsymbol{b}&  \boldsymbol{c}&  \boldsymbol{d}&  \boldsymbol{e}&  \boldsymbol{f}&  \boldsymbol{o}& P \\ \hline 0&0&1&0&1&1&1&0 \\ 1&0&0&1&0&1&1&0 \\ 1&1&1&0&0&0&1&0 \end{array}\right). 
\end{align*}
Thus we  get  $A_2({\bar D}_3')$ from $A_2({\bar D}_3)$ by integral elementary operations. 
\end{proof}

\begin{remark}
We easily show that  $A_1({\bar D}_1)\oplus (1)$ (resp. $A_1({\bar D}_3)$) 
and $A_1({\bar D}_1')$ (resp. $A_1({\bar D}_3')$) are $\mathbb{Z}$-equivalent in the same way. 
However, a similar argumet about $A_1({\bar D}_2)\oplus (1)\oplus (1)$ and $A_1({\bar D}_2')$ 
in the proof of Proposition \ref{prop:5-4} does not hold.  
\end{remark}
    
\begin{proof}[Proof of Theorem \ref{thm:5-2}]
We start with the knot projection ${\bar D}_0$ in Figure \ref{fig:5-1} 
and use Reidemeister moves for knot projections. 


The knot projection ${\bar D}_0$ has the smallest positive number of crossings. 
$A_2({\bar D}_0)=\begin{pmatrix}2&1&1\end{pmatrix}$ and clearly this matrix is $\mathbb{Z}$-equivalent to $E_{00}$ (of size $1\times 3$).  
As mentioned in Section $3$, any knot projection ${\bar D}$ can be obtained 
by Reidemeister moves as in Figure \ref{fig:5-2} from ${\bar D}_0$.  
\begin{figure}
\includegraphics[]{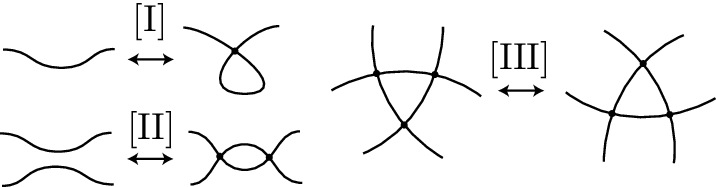}
\caption{}\label{fig:5-2}
\end{figure}
We have already showed 
in Proposition \ref{prop:5-3}, \ref{prop:5-4} and \ref{prop:5-5} 
that the property (being $\mathbb{Z}$-equivalent to $E_{00}$) of region choice matrices 
is preserved by these moves. 
Therefore this completes the proof. 
\end{proof}

\section{Appendix II}
Here is an example of the procedure of Propositon \ref{prop:3-3}.
\newpage

\begin{figure}
\includegraphics[]{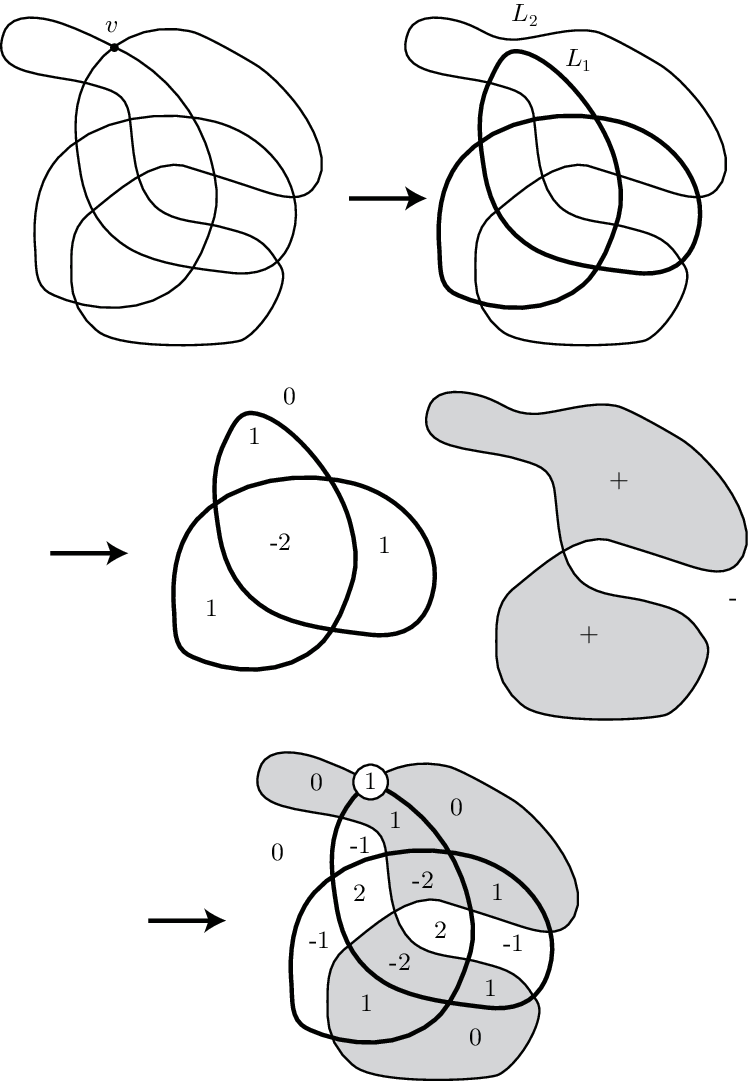}
\caption{}\label{fig:LLexample}
\end{figure} 

\newpage
@
\newpage
\section{Appendix III}

In this appendix 
we present a small table about region choice matrices of some knot projections.  
Each row consists of a picture of knot projection ${\bar D}$, 
the augmented matrix $\begin{pmatrix}A_1({\bar D}) & \boldsymbol{b}\end{pmatrix}$, 
and the echelon form by Gaussian elimination.

\par\bigskip
{\large $\bar{3_1}$}\mbox{}\\
\raisebox{-1.5cm}{\includegraphics[]{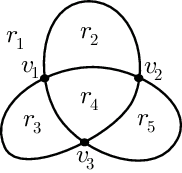}} \vspace{0.5cm}\qquad 
$\left(
\begin{array}{lllll|l}
 1 & 1 & 1 & 1 & 0 &b_1\\
 1 & 1 & 0 & 1 & 1 &b_2\\
 1 & 0 & 1 & 1 & 1 &b_3
\end{array}
\right) $
\\
$\left(
\begin{array}{ccccc|l}
 1 & 0 & 0 & 1 & 2 & -b_1+b_2+b_3 \\
 0 & 1 & 0 & 0 & -1 & -b_3+b_1 \\
 0 & 0 & 1 & 0 & -1 & -b_2+b_1
\end{array}
\right)$\\

\par\bigskip
{\large $\bar{4_1}$}\mbox{}\\
\raisebox{-1.5cm}{\includegraphics[]{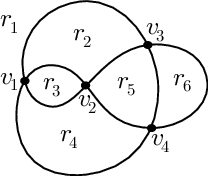}} \vspace{0.5cm}\qquad 
$\left(
\begin{array}{llllll|l}
 1 & 1 & 1 & 1 & 0 & 0 & b_1 \\
 0 & 1 & 1 & 1 & 1 & 0 & b_2 \\
 1 & 1 & 0 & 0 & 1 & 1 & b_3 \\
 1 & 0 & 0 & 1 & 1 & 1 & b_4
\end{array}
\right)$
\\
$\left(
\begin{array}{cccccc|l}
 1 & 0 & 0 & 0 & -1 & 0 & b_1-b_2 \\
 0 & 1 & 0 & 0 & 2 & 1 & -b_1+b_2+b_3 \\
 0 & 0 & 1 & 0 & -3 & -2 & 2 b_1-b_2-b_3-b_4 \\
 0 & 0 & 0 & 1 & 2 & 1 & -b_1+b_2+b_4
\end{array}
\right)$

\par\bigskip
{\large $\bar{5_1}$}\mbox{}\\
\raisebox{-1.5cm}{\includegraphics[]{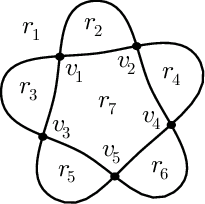}} \vspace{0.5cm}\qquad 
$\left(
\begin{array}{lllllll|l}
 1 & 1 & 1 & 0 & 0 & 0 & 1 & b_1 \\
 1 & 1 & 0 & 1 & 0 & 0 & 1 & b_2 \\
 1 & 0 & 1 & 0 & 1 & 0 & 1 & b_3 \\
 1 & 0 & 0 & 1 & 0 & 1 & 1 & b_4 \\
 1 & 0 & 0 & 0 & 1 & 1 & 1 & b_5
\end{array}
\right)$
\\
$\left(
\begin{array}{ccccccc|l}
 1 & 0 & 0 & 0 & 0 & 2 & 1 & b_1-b_2-b_3+b_4+b_5 \\
 0 & 1 & 0 & 0 & 0 & -1 & 0 & b_2-b_4 \\
 0 & 0 & 1 & 0 & 0 & -1 & 0 & b_3-b_5 \\
 0 & 0 & 0 & 1 & 0 & -1 & 0 & -b_1+b_2+b_3-b_5 \\
 0 & 0 & 0 & 0 & 1 & -1 & 0 & -b_1+b_2+b_3-b_4
\end{array}
\right)$

\par\bigskip
{\large $\bar{5_2}$}\mbox{}\\
\raisebox{-1.5cm}{\includegraphics[]{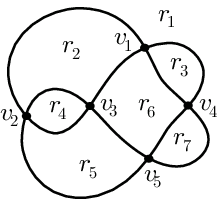}} \vspace{0.5cm}\qquad 
$\left(
\begin{array}{lllllll|l}
 1 & 1 & 1 & 0 & 0 & 1 & 0 & b_1 \\
 1 & 1 & 0 & 1 & 1 & 0 & 0 & b_2 \\
 0 & 1 & 0 & 1 & 1 & 1 & 0 & b_3 \\
 1 & 0 & 1 & 0 & 0 & 1 & 1 & b_4 \\
 1 & 0 & 0 & 0 & 1 & 1 & 1 & b_5
\end{array}
\right)$
\\
$\left(
\begin{array}{ccccccc|l}
 1 & 0 & 0 & 0 & 0 & -1 & 0 & b_2-b_3 \\
 0 & 1 & 0 & 0 & 0 & 0 & -1 & b_1-b_4 \\
 0 & 0 & 1 & 0 & 0 & 2 & 1 & -b_2+b_3+b_4 \\
 0 & 0 & 0 & 1 & 0 & -1 & 0 & -b_1+b_2+b_4-b_5 \\
 0 & 0 & 0 & 0 & 1 & 2 & 1 & -b_2+b_3+b_5
\end{array}
\right)$

\par\bigskip
{\large $\bar{6_1}$}\mbox{}\\
\raisebox{-1.5cm}{\includegraphics[]{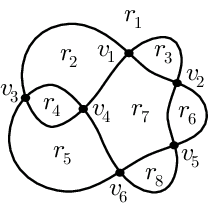}} \vspace{0.5cm}\qquad 
$\left(
\begin{array}{llllllll|l}
 1 & 1 & 1 & 0 & 0 & 0 & 1 & 0 & b_1 \\
 1 & 0 & 1 & 0 & 0 & 1 & 1 & 0 & b_2 \\
 1 & 1 & 0 & 1 & 1 & 0 & 0 & 0 & b_3 \\
 0 & 1 & 0 & 1 & 1 & 0 & 1 & 0 & b_4 \\
 1 & 0 & 0 & 0 & 0 & 1 & 1 & 1 & b_5 \\
 1 & 0 & 0 & 0 & 1 & 0 & 1 & 1 & b_6
\end{array}
\right)$
\\
$\left(
\begin{array}{cccccccc|l}
 1 & 0 & 0 & 0 & 0 & 0 & -1 & 0 & b_3-b_4 \\
 0 & 1 & 0 & 0 & 0 & 0 & 2 & 1 & b_1-b_2-b_3+b_4+b_5 \\
 0 & 0 & 1 & 0 & 0 & 0 & 0 & -1 & b_2-b_5 \\
 0 & 0 & 0 & 1 & 0 & 0 & -3 & -2 & -b_1+b_2+2 b_3-b_4-b_5-b_6 \\
 0 & 0 & 0 & 0 & 1 & 0 & 2 & 1 & -b_3+b_4+b_6 \\
 0 & 0 & 0 & 0 & 0 & 1 & 2 & 1 & -b_3+b_4+b_5
\end{array}
\right)$

\par\bigskip
{\large $\bar{6_2}$}\mbox{}\\
\raisebox{-1.5cm}{\includegraphics[]{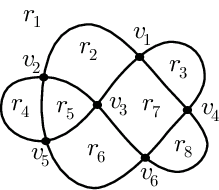}} \vspace{0.5cm}\qquad 
$\left(
\begin{array}{llllllll|l}
 1 & 1 & 1 & 0 & 0 & 0 & 1 & 0 & b_1 \\
 1 & 1 & 0 & 1 & 1 & 0 & 0 & 0 & b_2 \\
 0 & 1 & 0 & 0 & 1 & 1 & 1 & 0 & b_3 \\
 1 & 0 & 1 & 0 & 0 & 0 & 1 & 1 & b_4 \\
 1 & 0 & 0 & 1 & 1 & 1 & 0 & 0 & b_5 \\
 1 & 0 & 0 & 0 & 0 & 1 & 1 & 1 & b_6
\end{array}
\right)$
\\
$\left(
\begin{array}{cccccccc|l}
 1 & 0 & 0 & 0 & 0 & 0 & 1 & 2 & -b_1+b_2+b_4-b_5+b_6 \\
 0 & 1 & 0 & 0 & 0 & 0 & 0 & -1 & b_1-b_4 \\
 0 & 0 & 1 & 0 & 0 & 0 & 0 & -1 & b_1-b_2+b_5-b_6 \\
 0 & 0 & 0 & 1 & 0 & 0 & -2 & -3 & 2 b_1-b_2-b_3-2 b_4+2 b_5-b_6 \\
 0 & 0 & 0 & 0 & 1 & 0 & 1 & 2 & -2 b_1+b_2+b_3+2 b_4-b_5 \\
 0 & 0 & 0 & 0 & 0 & 1 & 0 & -1 & b_1-b_2-b_4+b_5
\end{array}
\right)$

\par\bigskip
{\large $\bar{6_3}$}\mbox{}\\
\raisebox{-1.5cm}{\includegraphics[]{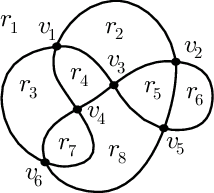}} \vspace{0.5cm}\qquad 
$\left(
\begin{array}{llllllll|l}
 1 & 1 & 1 & 1 & 0 & 0 & 0 & 0 & b_1 \\
 1 & 1 & 0 & 0 & 1 & 1 & 0 & 0 & b_2 \\
 0 & 1 & 0 & 1 & 1 & 0 & 0 & 1 & b_3 \\
 0 & 0 & 1 & 1 & 0 & 0 & 1 & 1 & b_4 \\
 1 & 0 & 0 & 0 & 1 & 1 & 0 & 1 & b_5 \\
 1 & 0 & 1 & 0 & 0 & 0 & 1 & 1 & b_6
\end{array}
\right)$
\\
$\left(
\begin{array}{cccccccc|l}
 1 & 0 & 0 & 0 & 0 & 0 & -1 & 0 & b_1-b_2-b_4+b_5 \\
 0 & 1 & 0 & 0 & 0 & 0 & 0 & -1 & b_2-b_5 \\
 0 & 0 & 1 & 0 & 0 & 0 & 2 & 1 & -b_1+b_2+b_4-b_5+b_6 \\
 0 & 0 & 0 & 1 & 0 & 0 & -1 & 0 & b_1-b_2+b_5-b_6 \\
 0 & 0 & 0 & 0 & 1 & 0 & 1 & 2 & -b_1+b_3+b_6 \\
 0 & 0 & 0 & 0 & 0 & 1 & 0 & -1 & b_2-b_3+b_4-b_6
\end{array}
\right)$

%% file: acknowledgements.tex

\section{acknowledgements}

The authors would like to thank Professor Akio Kawauchi 
and Professor Ayaka Shimizu 
for introducing them such an interesting problem.